\documentclass[11pt]{article}
\usepackage[left=1in,right=1in,bottom=1in,top=1in]{geometry}
\usepackage{amsmath}
\usepackage{amsfonts}
\usepackage{amssymb}
\usepackage{amsthm}
\usepackage{young}
\usepackage[vcentermath]{youngtab}
\usepackage{tikz}

\usepackage{algorithm}
\usepackage[noend]{algpseudocode}
\usepackage{amsmath, amssymb, amsthm}
\usepackage{blindtext}
\usepackage{hyperref} % Must load hyperref before cleverref
\usepackage{cleveref}
\usepackage{comment}
\usepackage{enumerate}
\usepackage{fullpage}
\usepackage{mathtools}
\usepackage{subcaption}

\makeatletter
\newtheorem*{rep@theorem}{\rep@title}
\newcommand{\newreptheorem}[2]{%
\newenvironment{rep#1}[1]{%
 \def\rep@title{#2 \ref{##1}}%
 \begin{rep@theorem}}%
 {\end{rep@theorem}}}
\makeatother

\parskip \medskipamount
\newtheorem{theorem}{Theorem}
\newtheorem{lemma}[theorem]{Lemma}
\newreptheorem{theorem}{Theorem}
\newtheorem{example}[theorem]{Example}
\newtheorem{remark}[theorem]{Remark}
\newtheorem{definition}[theorem]{Definition}

\title{Super--character theory and comparison arguments for a random walk on the upper triangular matrices}

\author{Evita Nestoridi}
\date{}

\begin{document}

\maketitle

\section{Introduction}
Recently, there has been a lot of interest in the mixing time of a specific random walk on upper triangular matrices (\cite{Andre2}, \cite{StanleyD}, \cite{Coppersmith}, \cite{Laurent}, \cite{Hough}, \cite{Pak}, \cite{Peres},  \cite{GM}, \cite{Stong}). Let $p$ be an odd prime and let $G$ be the group of $n \times n$ upper triangular matrices with $1$'s on the diagonal and elements of $\mathbb{Z}/p \mathbb{Z}$ above the diagonal. Let $E(i,i+1)$ be the  $n \times n$ matrix having a one  in the $(i,i+1)$ entry and zeros elsewhere. The set $S=\{I_n \pm E(i,i+1), 1\leq i \leq n-1\}$ is a symmetric generating set for $G$. We consider the random walk on $G$ using these generators, namely we let
%Define the following probability measure:
%\begin{equation}\label{initial}
%P(g)= \begin{cases} 
%\frac{1}{2(n-1)}, & \mbox{ if } g= I_n \pm E(i,i+1)    \\ 0, & \mbox{ otherwise.}
%\end{cases}
%\end{equation}
%We also write 
\begin{equation}\label{initial}
P_x(xg)= \begin{cases} 
\frac{1}{2(n-1)}, & \mbox{ if } g= I_n \pm E(i,i+1) ,   \\ 0, & \mbox{otherwise,}
\end{cases}
\end{equation}
be the probability of moving from $x$ to $xg$ in one step. The $t^{th}$ convolution of $P$ is defined inductively as
$$P^{*t}_x(y)= \sum_{w \in G} P^{*{(t-1)}}_w(y)P_x(w)$$
and it gives the probability of moving form $x$ to $y$ in $t$ steps. According to Proposition $2.13$ and Theorem $4.9$ of \cite{Peresbook}, the fact that $S$ is a symmetric set of generators guarantees that $P_{x}^{*t}$ converges to the uniform measure $\mu$ on $G$ with respect to the total variation distance, which is defined as
$$ ||P^{*t}_{x}- \mu||_{T.V.}  := \frac{1}{2} \sum_{g \in G} |P^{*t}_{x}(g)- \mu(g)|.$$
The main result of this paper concerns the mixing time of the  above walk with respect to the total variation distance, i.e.
$$t_{ mix}(\varepsilon)= \inf \lbrace t \in \mathbb{Z}  : \max_{x \in G} \{ ||P^{*t}_{x}- \mu||_{T.V.} \} < \varepsilon \rbrace,$$
where $\mu$ is the uniform measure on $G$. 
\begin{theorem}\label{main}
There exist universal constants $0<b,d< \infty$ such that  for $c>0$ and $t\geq c b p^2n^4$, we have that
$$4||P^{*t}- \mu||^2_{T.V}\leq  d e^{-c},$$
for $p$ sufficiently large.
\end{theorem}
The dependence on $p$ is the best possible, since the entry in position $(n-1,n)$ performs a simple random walk on $\mathbb{Z}/p \mathbb{Z}$ and the mixing time of this random walk is of order $p^2$ as explained in \cite{Threads}.

The proof of Theorem \ref{main} relys on bounding the eigenvalues $1=\lambda_0 > \lambda_1 \geq \ldots \geq \lambda_{|G|-1}>-1 $ of the transition matrix $:=(P_x(y))_{x,y \in G}$ and using the inequality 
$$4||P^{*t}- \mu||^2_{T.V}\leq \sum_{i=1}^{|G|-1} \lambda^{2t}_i$$
(see Lemma 12.16 (ii)of \cite{Peresbook}).
To bound the eigenvalues of $P$, we introduce an auxiliary  random walk $Q$ on $G$, which we study using the super--character theory of $G$. Then, we use comparison theory as introduced by Diaconis and Saloff-Coste \cite{Compar} to bound the eigenvalues of $P$ in terms of the eigenvalues of $Q$.

This new walk $Q$ is defined as follows. Let $$a=\begin{cases} \lfloor \sqrt{p} \rfloor, & \mbox{ if } \lfloor \sqrt{p} \rfloor \mbox{ is odd,}\\
 \lfloor \sqrt{p} \rfloor +1, & \mbox { otherwise,} \end{cases}$$
 be the closest odd integer to $\lfloor \sqrt{p} \rfloor$.
Define the following probability measure on $G$:
\begin{equation}\label{random_2}
Q_x(xg)= \begin{cases} \frac{1}{4(n-1)p^{n-2}}, & \mbox{ if } g \in C_i(\pm 1) \cup C_i(\pm a)    \\ 0, & \mbox{ otherwise,}
\end{cases}
\end{equation}
where $C_i(\pm 1)$ denotes the conjugacy class of the matrix $I_n \pm E(i,i+1)$ and $ C_i(\pm a) $ denotes similarly  the conjugacy class of the matrix $I_n \pm  a E(i,i+1)$.

 \begin{theorem}\label{rootp}
 There exist uniform constants $0<\alpha, \beta< \infty $ such that for $c>0$ and $t= c \beta p n \log n$, then
 $$4||Q^{*t}- \mu||^2_{T.V.} \leq  \alpha e^{-c} .$$
 \end{theorem}

%\begin{remark}
%The second part of Theorem \ref{main} says that order $n^4p^2$ steps suffice to reach stationarity. This is the main contribution of this paper; providing upper bounds for the mixing time with the correct dependency
%on $p$. The first part of Theorem \ref{main} says that better dependency on $n$ can be achieved if $n$ is sufficiently small. That is if $n \log p  << c \beta p  \log n $, which can happen for $n$ small and $p$ big,  then order $p^2n^3 \log n$ steps suffice to reach stationarity.
%\end{remark}
Section \ref{lit} gives details on the rich literature of this problem. Sections \ref{setup} and \ref{supchar} provide a quick overview of the super--character theory needed. In Section \ref{fourier}, we present a Fourier analysis argument, which leads to the proof of Theorem \ref{rootp}, contained in Section \ref{proofrootp}. Section \ref{comparison} provides a brief review of the comparison techniques introduced by Diaconis and Saloff-Coste \cite{Compar} and then uses them to prove Theorem \ref{main}.

\begin{remark}
The case of $p=2$ of the walk we consider has been thoroughly studied by Peres and Sly \cite{Peres}, who proved an upper bound for the mixing time of order $n^2$ and Stong \cite{Stong} who proved a lower bound also of order $n^2$.
\end{remark}
\section{Literature}\label{lit}
Many people have studied similar problems, starting with Zack \cite{Zack},  who was interested in the Heisenberg group (which is $G$ for the case $n=3$).
Diaconis and Saloff-Coste \cite{Laurent} used Nash inequalities to prove that  for the walk on $F$ for the case where $n$ is fixed and $p $ large, the mixing time is bounded above and below by a constant times $p^2$. See Diaconis and Hough \cite{Hough} for a broader review of the $n=3$ case and the extensions to nilpotent groups.

Stong \cite{Stong} found sharp bounds for the second and last eigenvalues of the $P-$walk which allowed him to prove an upper bound of order $p^2 n^3 \log p$. He also shows that at least $n^2$ steps are needed for the $P-$walk. Arias-Castro, Diaconis and Stanley \cite{StanleyD} then used super--character theory and comparison theory to give an upper bound of order $n^4 p^2 \log n + n^3 p^4 \log n$, taking into consideration Stong's earlier bounds on the eigenvalues \cite{Stong}. They prove it by doing \emph{super--character} theory analysis for the walk generated by $C_i(\pm 1)$ and, then, doing comparison theory \cite{Compar}. They also prove a lower bound of the form $p^2 n \log n$. 

Coppersmith and Pak  (\cite{Pak}, \cite{Coppersmith}) looked at the walk generated by $\{ g= I_n + a E(i,i+1), a \in \mathbb{Z}/p \mathbb{Z}\}$  and managed to improve the $n$ term of the upper bound in the case where $n>> p^2 $. They proved that $n^2 \log p$ steps are sufficient to reach stationarity. Peres and Sly \cite{Peres} proved that for $p=2$ the sharp bound is of order $n^2$  using the east model.  We refer to Peres and Sly for a more complete  survey of the existing literature. 

The present paper depends a lot on works by Andr\' e (\cite{Andre3}, \cite{Andre}, \cite{Andre2}), Carter and Yan \cite{Yan}. They have developed a theory using certain unions of conjugacy classes, that we will refer to as super-classes, and sums of irreducible characters, that are called \emph{super--characters}. Our work sharpens the super--character theory techniques introduced by Arias-Castro, Diaconis and Stanley and fixes the dependency
on $p$.

\section{Preliminaries}
We first consider the following random walk on $\mathbb{Z}/p\mathbb{Z}$:
$$K_x(y):= \begin{cases} 
\frac{1}{4}, & \mbox{ if } y=x \pm a ,  \\
\frac{1}{4}, & \mbox{ if } y=x \pm 1 ,
 \\ 0, & \mbox{otherwise.}
\end{cases}$$
This random walk, is a special case of $Q$ for the case $n=2$.
Theorem 6 of \cite{Pbook} says that the eigenvalues of the matrix $K$ are given by the Fourier transform of the irreducible representations of  $\mathbb{Z}/p\mathbb{Z}$ with respect to $K(g)=\begin{cases} 
\frac{1}{4}, & \mbox{ if } g= \pm a , \pm 1   \\ 0, & \mbox{otherwise} 
\end{cases}$.
In particular, if $\rho_x(j)= e^{\frac{2 \pi i x j}{p}}$ for $x, \j \in \mathbb{Z}/p\mathbb{Z}$, then $\hat{K}(\rho_x):=\sum_{y= \pm 1, \pm a} K(y)\rho(y) \frac{1}{2}\cos \frac{2 \pi  x a}{p}+ \frac{1}{2}\cos \frac{2 \pi  x }{p}$ are the eigenvalues of the transition matrix $(K_x(y))_{x,y \in \mathbb{Z}/p \mathbb{Z}}$. 
\begin{lemma}\label{cycle}[\cite{Threads}, Example 2.3]
\begin{enumerate}
\item[(a)] We have that 
$$\Vert K_{y}^t - U \Vert_2^2= \sum_{x =1}^{p-1} \left(\frac{1}{2}\cos \frac{2 \pi  x a}{p}+ \frac{1}{2}\cos \frac{2 \pi  x }{p}\right)^{2t},$$
where $y \in \mathbb{Z}/p\mathbb{Z}$ and $ U$ is the uniform measure of $\mathbb{Z}/p\mathbb{Z}$.
\item[(b)] If $t=c p$, then 
$$\Vert K_{y}^t - U \Vert_2^2\leq A e^{-c},$$
where $y \in \mathbb{Z}/p\mathbb{Z}$ and $A$ is uniform constant.
\end{enumerate}
\end{lemma}
\begin{proof}
Part (a) follows from Lemma 12.18 of \cite{Peresbook} and the fact that $\frac{1}{2}\cos \frac{2 \pi  x a}{p}+ \frac{1}{2}\cos \frac{2 \pi  x }{p}$ are the eigenvalues of $K$. Part (b) follows by the analysis presented in example 2.3 of \cite{Threads}.
\end{proof}

The following lemma is a key computation for the proof of Theorem \ref{rootp}.
\begin{lemma}
We have that there are $\alpha, \beta, $ uniform constants in $p,n$, such that if $t=  \beta p N \log N$
\begin{equation}
\sum_{(x_1,\ldots, x_n) \in (\mathbb{Z}/p\mathbb{Z})^N \setminus \{\textbf{0}\}}  \left( \frac{1}{2N} \sum^{N}_{j=1} \left( \cos \left(\frac{2 \pi x_ja }{p} \right)  +  \cos \left(\frac{2 \pi x_j }{p}\right) \right)  \right)^{2t} \leq \alpha e^{-\beta}.
\end{equation}
\end{lemma}

\begin{proof}
Consider the following random walk on $(\mathbb{Z}/p\mathbb{Z})^N$:
$$q(x,y)= \begin{cases} 
\frac{1}{4N}, & \mbox{ if } y= x\pm a e_i ,x\pm  e_i \mbox{ for } i=1,\ldots,N   \\ 0, & \mbox{otherwise,} 
\end{cases}$$
where $e_i$ is the vector in $(\mathbb{Z}/p\mathbb{Z})^N$ that has a one in the $i$th position and everywhere else zero.

Theorem 6 of chapter 3 of \cite{Pbook} says that the eigenvalues of $q$ are indexed by $(x_1, \ldots, x_{N}) \in (\mathbb{Z}/p\mathbb{Z})^N $ and they are equal to
\begin{equation}
   \frac{1}{2N} \sum^{N}_{j=1} \left( \cos \left(\frac{2 \pi x_ja }{p} \right)  +  \cos \left(\frac{2 \pi x_j }{p}\right) \right)
\end{equation}
Lemma \ref{cycle} and Theorem 1 of Section 5 of  Diaconis and Saloff-Coste \cite{Compar} says that 
 $t=  \beta p N \log N$
\begin{equation*}
\sum_{(x_1,\ldots, x_N) \in (\mathbb{Z}/p\mathbb{Z})^N \setminus \{\textbf{0}\}}  \left( \frac{1}{2N} \sum^{N}_{j=1} \left( \cos \left(\frac{2 \pi x_ja }{p} \right)  +  \cos \left(\frac{2 \pi x_j }{p}\right) \right) \right)^{2t} = \Vert q_{y}- \pi \Vert^2_2 \leq \alpha e^{-\beta},
\end{equation*}
where $\pi $ is the uniform measure on $(\mathbb{Z}/p\mathbb{Z})^N $. 

\end{proof}

\section{The conjugacy classes and the super--classes}\label{setup}
While a description of general conjugacy classes (and characters in $G$) remains unknown \cite{GKPRT}, as explained in \cite{StanleyD}, there is  a full description of the conjugacy class of $I_n + x E(i,i+1)$, where $x \neq 0$. It consists of all matrices in $G$ whose $(i,i+1)$ entry is $x$, the entries of the $i+1$ column exactly above $(i,i+1)$ are arbitrary elements $a_1,a_2,...,a_{i-1}$, the entries of the $i^{th}$ row exactly to the right of $(i,i+1)$ are arbitrary elements $b_1,...,b_{n-i-1}$ and  in the block surrounded by these $a_j, b_k$ the $(j,k)$ entry is $x^{-1} a_jb_k$.
Here is an example for $n=6$:
$$\left[ \begin{matrix}
&1 & 0 & 0 & a_1 & a_1 b_1x^{-1} & a_1 b_2 x^{-1}\\
&0 & 1  &0&  a_2 & a_2 b_1 x^{-1} & a_2 b_2 x^{-1}\\
&0 & 0 & 1 & x & b_1 & b_2\\
&0 & 0 & 0 & 1 & 0 & 0\\
&0 & 0 & 0 &  0 & 1 & 0 \\
& 0 &0 & 0 & 0 &  0 & 1 
\end{matrix}  \right]$$
Despite the fact that the character theory of $G$ is unknown, Andr\' e, Carter and Yan have provided a formula for the \emph{super--characters} of $G$ (certain specific sums of irreducible characters). Arias-Castro, Diaconis and Stanley \cite{StanleyD} showed how to use \emph{super--characters} to prove upper bounds for mixing times. This will be explained in Section \ref{fourier}, after establishing the necessary notation and terminology.

Here is the general description of the \emph{super--classes} of $G$ (certain unions of conjugacy classes). 
Let $U_n(p)$ be the set of all $n \times n$ upper triangular  matrices with zeros on the diagonal and let $G \times G$ act on $U_n(p)$ by right and left multiplication. Let $\Psi$ denote the set of orbits of this action, which we refer to as \emph{transition} orbits. According to Yan, each \emph{transition} orbit contains a unique representative which has at most one non-zero entry per column and per row (see Theorem $3.1$ of \cite{Yan2}). Thus $\Psi$ consists of pairs $(D, \phi)$, where $D$ is a collection of positions $(i,j)$ with $i<j$ at most one per column and per row and $\phi : D \rightarrow F^*_p$ any map.

A \emph{super--class} in $G$ corresponds to a \emph{transition} orbit and consists of those elements of $G$ of the form $I_n$ plus an element of the \emph{transition} orbit. Yan explains at the end of Section $2$ of \cite{Yan2} that the \emph{super--class} of an element of the form $I_n + x E(i,i+1)$ in facts coincides with its conjugacy class. For $(D, \phi)\in \Psi$, denote the corresponding \emph{super-class} by $C(D, \phi)$.

\section{The super--characters}\label{supchar}
The \emph{super--characters} are certain sums of characters that can be used to bound the mixing time of the walk generated by $Q$, as it is later described in Lemma \ref{upper bound lemma}. Here is the description of the \emph{super--characters} as provided by Yan \cite{Yan}.

Let $U^*_n(F_p)$ be the space of linear maps from $U_n(p)$ to $F_p$. Then $G$ acts on the left and right of $U^*_n(p)$ as follows:
$$(g \star \lambda)(m)= \lambda(mg), \quad (\lambda \star g) (m)= \lambda (gm),$$
where $g \in G, \lambda \in U^*_n(F_p)$ and $m \in U_n(p) $. The orbits of the action of $G \times G$ on $U^*_n(F_p)$ are called \emph{cotransition}  orbits.

The left action gives the regular representation of $G$ on $\mathbb{C}[G]$. To get an element of the group algebra $\mathbb{C}[G]$, we consider a non-trivial homomorphism $\theta:F_p \rightarrow \mathbb{C}^* $ from the additive group $F_p$ to the non-zero complex numbers. Then, for $\lambda \in U^*_n(F_p)$ we get the element of the group algebra $u_{\lambda} : G \rightarrow \mathbb{C}$ defined as:
 $$u_{\lambda}(g)= \theta(\lambda(g-I))$$
The goal is to decompose regular representation of $G$ on $\mathbb{C}[G]$ into a sum of orthogonal submodules of $\mathbb{C}[G]$ (not necessarily irreducible). 
Proposition 2.1 of \cite{Yan} says that
$$g \cdot u_{\lambda}= u_{\lambda}(g)  u_{g \cdot \lambda}$$
Therefore, if $L$ is an orbit of the left action of G on $U^*_n(F_p),$ then $span \lbrace  u_{\lambda}\rbrace_{ \lambda \in L }$ is a submobule of $\mathbb{C}[G]$. Corollary 2.3 of  \cite{Yan} says that the character $\chi_{\lambda}$ only depends on the \emph{contransition} orbit to which $\lambda$ belongs to. Theorem 3.2 of \cite{Yan} adds that the \emph{cotransition} orbits are indexed by pairs $(D, \phi)$ where $D$ denote the positions of the non-zero entries and $\phi : D \rightarrow F^*_p$ is the map that assigns a non-zero entry to each $(i,j)$ of $D$.
  
Let $\Psi^*$ denote the set of orbits of the action of G on $U^*_n(F_p)$ and $\chi_{D, \phi}$ be the character corresponding to the above representation, where $D$ and $ \phi$ determine the conjugacy class 
 we described in Section  \ref{setup}. Proposition 2.2 of \cite{Yan} proved that if $\lambda$ and 
 $\lambda'$ are in the same right orbit of $G$ acting on $U^{*}_n(F_p)$ then 
 $\chi_{\lambda}=\chi_{\lambda'}$ and therefore it makes sense to talk about  $\chi_{D, \phi}$. Also, Corollary 2.8 of \cite{Yan} says that  $\lbrace \chi_{\lambda} \rbrace_{ \lambda \in \Psi^*}$ are orthogonal characters.
 
%Crucially, a combination of work of Andr\' e,  Carter, Yan, Arias-Castro, Diaconis and Stanley for the case of \emph{super--class} functions gives an explicit, reasonably simple formula for the the value of a \emph{super--character} and a \emph{super--class}. This is described in section \ref{proofrootp} for the case of interest here.  
 
\section{Fourier transform setup}\label{fourier}
For $Q$ a probability measure on $G$ which is also a \emph{super--class} function, which means that $Q$ is constant on the \emph{super--classes}, and $\chi_{D, \phi}$, let 
\begin{equation}\label{f}
\widehat{Q}(D, \phi)= \sum_{g \in G} \chi_{D,\phi}(g) Q(g),
\end{equation}
denote the Fourier Transform of  $Q$ at $\chi_{D,\phi}$. Then, E.\ Arias-Castro, P.\ Diaconis and R.\ Stanley \cite{StanleyD} proved the following upper bound lemma using Fourier Transform arguments:
 
 \begin{lemma}\label{upper bound lemma}[\cite{StanleyD}, Proposition 2.4]
We have that
 \begin{equation}\label{l2}
 4||Q^{*t}- \mu||^2_{T.V.} \leq \sum_{D \neq \emptyset, \phi} p^{-i(D)} \left( \frac{\widehat{Q}(D, \phi)}{p^{d(D)}} \right)^{2t} ,
 \end{equation}
 where $d(D)$ is the sum of the vertical distances from the boxes in $D$ to the super diagonal $\lbrace (i,i+1)\rbrace_{1 \leq i \leq n-1}$ and  $i(D)$ counts the number of pairs of boxes $(i,j), (k,l)$ in $D$ with $1\leq i<k<j<l \leq n$ so that the corner $(k,j)$ is strictly above the diagonal.
 \end{lemma}

\begin{remark}
The statistics $i(D)$ and $d(D)$ are discussed in full detail in Yan \cite{Yan}, Arias-Castro, Diaconis and Stanley \cite{StanleyD}.
\end{remark}

\section{Proof of Theorem \ref{rootp}}\label{proofrootp}
The proof of Theorem \ref{rootp} follows the proof of Theorem $1.1$ of \cite{StanleyD} and makes use of Lemma \ref{upper bound lemma}:
\begin{proof}
To bound the right hand side of \eqref{l2}, we consider each summand $\chi_{D,\phi}(g)$ of \eqref{f}. For any $D$ let $D_i$ be  the set of positions in D in the rectangle strictly above and to the right of $(i,i+1)$ and $R(\{i,i+1\})$ is the complement of all the positions that are directly below $(i,i+1 )$ or directly to the right of $(i,i+1)$. Then the formula of Remark 2 of Section 2.3 of Arias-Castro, Diaconis and Stanley says that
\begin{equation}\label{1}
\frac{\chi_{D,\phi}(g)}{p^{d(D)}}= \begin{cases} p^{-|D_i| }\theta(\pm \phi(i,i+1)), & \mbox{if } D \subset R(\{i,i+1\}) \\ 0, &  \mbox{ otherwise,} \end{cases},
\end{equation}
if $g \in C_i(\pm 1)$ and 
\begin{equation}\label{2}
\frac{\chi_{D,\phi}(g)}{p^{d(D)}}= \begin{cases} p^{-|D_i|} \theta(\pm a \phi(i,i+1)), & \mbox{if } D \subset R(\{i,i+1\}) \\ 0, & \mbox{ otherwise,} \end{cases}
\end{equation}
if $g \in C_i(\pm a)$. The right hand side of \eqref{l2} can be bounded as follows 
$$\sum_{D\neq \emptyset, \phi} p^{-i(D)} \left(\frac{\widehat{Q}(D, \phi)}{p^{d(D)}}\right)^{2t}  \leq \sum_{D\neq \emptyset, \phi} \left(\frac{\widehat{Q}(D, \phi)}{p^{d(D)}}\right)^{2t}.$$
Equations \eqref{f}, \eqref{1} and \eqref{2} give that 
$$\frac{\widehat{Q}(D, \phi)}{p^{d(D)}} = \frac{1}{2n-2} \sum^{2n-2}_{i=1} w_i(D) \left(\cos \frac{2 \pi \phi(i,i+1)}{p} + \cos \frac{2 \pi \phi(i,i+1)a}{p}\right)$$
where the weights $w_i (D)$ satisfy $0 \leq w_i(D) \leq 1$ and $w_i(D)=0$ whenever there is $s$ such that $(i,s) \in D$ or $(s,i+1) \in D$.
Let  $Z(D)$ be the set of $i=1,2,...,n-1$ such that $w_i(D)=0$. 

For $x \in \mathbb{Z}/ p \mathbb{Z}$, let $I_x^{+}(\phi)$ (respectively $I_{x}^{-}(\phi)$)  be the set of   
$i=1,2,3,...,n-1$ such that $\cos \frac{2 \pi \phi(i,i+1)x}{p}>0$ (resp. 
$\cos \frac{2 \pi \phi(i,i+1)x}{p}<0$). 
The following will be the dominating terms:

\begin{equation*}
A^{+}(D,\phi)=\frac{1}{2n-2} \left( \sum_{\substack{ i \in \\I_1^{+}(\phi)\cap Z(D)^c}} \cos \frac{2 \pi \phi(i,i+1)}{p} + \sum_{\substack{ i \in \\I_{a}^{+}(\phi)\cap Z(D)^c}} \cos \frac{2 \pi \phi(i,i+1)a}{p}  \right)
\end{equation*}
and 
\begin{equation*}
 A^{-}(D,\phi)=\frac{1}{2n-2} \left( \sum_{\substack{ i \in \\I_1^{-}(\phi)\cap Z(D)^c}} \cos \frac{2 \pi \phi(i,i+1)}{p} + \sum_{\substack{ i \in \\I_{a}^{-}(\phi)\cap Z(D)^c}} \cos \frac{2 \pi \phi(i,i+1)a}{p}  \right)
 \end{equation*}
Since 
$$A^{-}(D,\phi) \leq \frac{\widehat{Q}(D, \phi)}{p^{d(D)}}  \leq A^{+}(D,\phi),$$
$$\sum_{D\neq \emptyset, \phi} \left(\frac{\widehat{Q}(D, \phi)}{p^{d(D)}}\right)^{2t} \leq S^{+} + S^{-}$$
where $$ S^{\pm} =\sum_{D\neq \emptyset, \phi} \left( A^{\pm}(D,\phi) \right)^{2t}  $$
To bound $S^{+}$:

Let $b(D)$ be the cardinality of the elements of D that are off the super diagonal and $c^{\pm}(D)=|I_{a}^{\pm}(\phi) \cap on(D)|+ | I_{1}^{\pm}(\phi)\cap on(D)|$ where $on(D)$ are the elements of the super diagonal of D.

Then, replacing $\phi(i,i+1)$ by $h_i$,
\begin{align*}
&S^{+} \leq 
 \sum_D p^{b(D)} p^{ c^{-}(D)} \left( \frac{c^+(D)}{2n-2} \right)^{2t}\\
& \sum_{h_1,h_2,...,h_{c^+(D)}}  \left( \frac{1}{c^+(D)} \left( \sum_{\substack{ i \in \\I_p^{+}(\phi)\cap Z(D)^c}} \cos \frac{2 \pi h_i}{p} + \sum_{\substack{ i \in \\I_{a}^{+}(\phi)\cap Z(D)^c}} \cos \frac{2 \pi h_ia}{p}  \right) \right)^{2t} 
\end{align*}
Lemma \ref{cycle} says that after $t=cp n \log n $ steps, there are uniform constants $\alpha$ and $\beta$ such that
\begin{equation}\label{first}
\sum_{h_1,h_2,...,h_{c^+(D)}}  \left( \frac{1}{c^+(D)} \left( \sum_{\substack{ i \in \\I_p^{+}(\phi)\cap Z(D)^c}} \cos \frac{2 \pi h_i}{p} + \sum_{\substack{ i \in \\I_{a}^{+}(\phi)\cap Z(D)^c}} \cos \frac{2 \pi h_ia}{p}  \right) \right)^{2t} \leq \alpha e^{- \beta c}.
\end{equation}

We, also, need to bound the term $T=\sum_D p^{b(D)} p^{ c^{-}(D)} \left( \frac{c^+(D)}{2n-2} \right)^{2t}$. Following the second half of the proof of Theorem 1.1 of \cite{StanleyD}, for $t=2n (p+2) \log n +dn$
$$ T\leq   1+ 2 e^{-d}$$
where $d>0$.
To prove this, notice that if $a(D) $ is the cardinality of $Z(D)$ then $a(D)+ c^+(D) + c^-(D) \leq 2n-2$ and $a(D) > b(D)$ so
$$T \leq \sum_D p^{b(D)+ c^-(D)} \left( 1-\frac{b(D)+c^-(D)}{2n-2} \right)^{2t} $$
since there are at most ${n-1 \choose c} \times {n^2 \choose c}$ sets of positions with $b=b(D)$ and $c=c^-(D)$, which is bounded by $n^{2(b+c)}$ 
$$T \leq 1 + \sum_{1\leq b+c \leq n-1} (np)^{2(b+ c)} \left( 1-\frac{b+c}{2n-2} \right)^{2t} \leq$$
$$1+ \sum^n_{l=1} (pn)^{2l} e^{-tl/ (n-1)}$$
For $t>2n p \log n +dn$, we have that 
$$T\leq 1+ e^{-d}\sum^{n-1}_{l=1} \left( \frac{p}{n^p}\right)^{2l} \leq 1+ \frac{e^{-d}}{1-\frac{p}{n^{p+1}}} \leq 1+ 2e^{-d} $$
since $n^{p+1} \geq 2 p$ for $n>2$.

Therefore, overall there are new $\alpha, \beta, $ uniformly in $p,n$ such that for $c>0$ such that if $t= c \beta p n \log n$
$$
 \sum_D p^{b(D)} p^{ c^{-}(D)} \left( \frac{c^+(D)}{2n-2} \right)^{2t}
\sum_{h'_1,h'_2,...,h'_{c^+(D)}}  \left( \frac{1}{c^+(D)}  \sum^{c^+(D)}_{ i=1} \cos \frac{2 \pi h'_ia}{p}   \right)^{2t}  \leq \alpha e^{-c}
$$

 Similar arguments can be used to bound $S^-$.

\end{proof}

\section{The Comparison Argument}\label{comparison}
A comparison argument allows to use theorem \ref{rootp} in order to prove theorem \ref{main}. The $L^2$ distance of $P^{*t}$ from $\mu$ is defined as
$$||P^{*t}- \mu||_2:= \left( \sum_{g \in G} |P^{*t}(g)- \mu(g)|^2\right)^{1/2} .$$
The Cauchy-Schwartz inequality gives that
 $$4||Q^{*t}- \mu||^2_{T.V.} \leq |G|||P^{*t}- \mu||^2_2.$$
A direct application of Lemma 8 of \cite{Compar} gives the following tool.
\begin{lemma}\label{comp}
Let $P,Q$ be the probability measures on $G$, that were defined in the introduction. Let $ g \in C_i(\pm 1) \cup C_i(\pm a) $. Fix a way of writing $g$ as a product of the generators $\{ I_n \pm E(i,i+1)\}$ of odd length. Let $|g|$ be the length of this word and if $z \in \{ I_n \pm E(i,i+1)\}$ and let $N(g,z)$ be the number of times $z$ appears in this word. Then
\begin{equation*}
|G|||P^{*t}- \mu||^2_2 \leq |G| (e^{-t/A}+ ||Q^{*t/2A}- \mu||^2_2 )
\end{equation*}
where $$A=\max_{\substack{z \in \\  \{I_n \pm E(i,i+1) \} } } \frac{1}{P(z)} \sum_{\substack{g  \in  \\ C_i(\pm 1) \cup C_i(\pm a) }} |g|N(g,z) Q(g).$$ 
\end{lemma}
Notice that each element of $C_i(\pm 1) \cup C_i(\pm a)$ must be expressed as a product of elements of the form $\{ I_n \pm E(i,i+1)\}$. It so happens that the paths considered in the following section are of odd length exactly because $a$ and one are odd integers.

\subsection{Building up $I_n+bE(i, i+2)$ in $O(\sqrt{b})$ steps}

At first, the goal is to create the element $I+E(i, i+2)$ which has the entry $1$ in position $(i, i+2)$. But that simply occurs by considering the following commutator:
$$ I+E(i, i+2)=[I+E(i+1, i+2), I-E(i,i+1)],$$
where $[x,y]= x^{-1}y^{-1}xy$ for $x,y \in G$.
For example, if $n=4$ and $i=2$, we have that:
$$
 \left[ \begin{matrix}
&1 & 0 & 0 & 0\\
&0 & 1 & 0 & 1\\
&0 & 0 & 1 & 0\\
&0 & 0 & 0 & 1
\end{matrix}  \right]=$$
$$
\left[ \begin{matrix}
&1 & 0 & 0 & 0\\
&0 & 1 & 0 & 0\\
&0 & 0 & 1 & -1\\
&0 & 0 & 0 & 1
\end{matrix} \right]
\left[ \begin{matrix}
&1 & 0 & 0 & 0\\
&0 & 1 & 1 & 0\\
&0 & 0 & 1 & 0\\
&0 & 0 & 0 & 1
\end{matrix} \right]
\left[ \begin{matrix}
&1 & 0 & 0 & 0\\
&0 & 1 & 0 & 0\\
&0 & 0 & 1 & 1\\
&0 & 0 & 0 & 1
\end{matrix} \right]
\left[ \begin{matrix}
&1 & 0 & 0 & 0\\
&0 & 1 & -1 & 0\\
&0 & 0 & 1 & 0\\
&0 & 0 & 0 & 1
\end{matrix} \right].
$$
Then just notice the following identity holds if the entries are over $\mathbb{R}$:
\begin{equation}\label{almost}
 I+ b E(i,i+2)=  [I+ \sqrt{b} E(i+1,i+2),I- \sqrt{b} E(i,i+1) ].
\end{equation}
This idea gives rise to the following lemma.
\begin{lemma}\label{root}
We can express $I_n+bE(i, i+2)$ as a word in the generators $S $ whose length is even and is at most $12\lfloor\sqrt{b} \rfloor +10$.
\end{lemma}
\begin{proof}
Equation \eqref{almost} is over $\mathbb{R}$. To get something similar over $\mathbb{Z}/p \mathbb{Z}$, consider the following identity 
\begin{align}\label{actual}
I+ bE(i,i+2)&= (I+(b- \lfloor\sqrt{b}\rfloor^2)E(i, i+2)) (I+ \lfloor\sqrt{b}\rfloor^2 E(i,i+2)),
\end{align}
where we use \eqref{almost} to get $E(i, i+2)) (I+ \lfloor\sqrt{b}\rfloor^2 E(i,i+2))$ and we write
$$(I+(b- \lfloor\sqrt{b}\rfloor^2)E(i, i+2))=[ (I +E(i, i+1) (I +E(i, i+2) )^{b- \lfloor\sqrt{b}\rfloor^2}].$$
The length of the word that occurs by \eqref{almost} and \eqref{actual} is at most $4(b- \lfloor\sqrt{b}\rfloor^2) + 4\lfloor\sqrt{b} \rfloor \leq 12\lfloor\sqrt{b} \rfloor +10 $. This means that we can achieve to express $I_n+bE(i, i+2)$ as a word in the generators $ S $, whose length is at most $O(\sqrt{b})$.  Notice that we expressed $I+ bE(i,i+2)$ as a product of one or two commutators, therefore the length of the word is even.
\end{proof}

\subsection{Building up $C_i(\pm 1)$}\label{first paths}
In this section, we show how to express an element of $C_i(\pm 1)$ as a word in the generators $S$.
\begin{lemma}\label{C_i(1)}
Let $B \in C_i(\pm 1)$. Then, $B$ can be expressed as a word in the generators $S$, whose length is odd and is at most $O(n\sqrt{p})$. Also, each generator $I \pm E(j,j+1)$ appears at most $O(\sqrt{p})$ times in such a word.
\end{lemma}

To build an element $B$ of the conjugacy class of $I \pm E(i,i+1)$, the main idea is to build two matrices $B_1$ and $B_2$ in $G$ whose product is $B$. More specifically, the $i+1, i+3, i+5 \ldots $ columns of  $B_1$ are the same as the ones of $B$ and the rest zeros and similarly $B_2$ has the same $i+2,i+4, \ldots$ column as $B$ and the rest zeros. $B_1$ will have odd length when expressed as a word in the generators $S$, while $B_2$ will have even length. This way, $B$ will have indeed odd length, something that is needed to do comparison.

\paragraph*{Building up the odd columns.}
This section describes how to build $B_1$. 
\begin{definition}\label{A}
Let $A_i$ be the matrix that has ones on the diagonal and on the positions $(i,i+1), (i-1,i+1), \ldots (1,i+1)$.
\end{definition}
For example, if $n=6$ and $i=3$, we have that
\begin{align*}
&A_3= \left[ \begin{matrix}
&1 & 0 & 0 & 1 & 0 & 0\\
&0 & 1 & 0 & 1 & 0 & 0\\
&0 & 0 & 1 & 1 & 0 & 0\\
&0 & 0 & 0 & 1 & 0 & 0\\
&0 & 0 & 0 &  0 & 1 & 0 \\
& 0 &0 & 0 & 0 &  0 & 1 
\end{matrix}  \right].
\end{align*}
Firstly, notice that to express $A_i$ as a word in the generators $S$, conjugate $I+E(i,i+1)$ by $I-E(i-1,i)$ to get a one exactly above the position $(i,i+1)$. Continue conjugating by $I-E(j-1,j), j=2,3,\ldots i-1$ to get ones everywhere above the position $(i,i+1)$. This can be formally written as 
\begin{equation}\label{conA}
A_i=\left( \prod_{j=1}^{i-1} (I+E(j,j+1)) \right) (I+E(i,i+1)) \left( \prod_{j=1}^{i-1} (I-E(i-j,i-j+1))\right),
\end{equation}
where $\prod_{i=1}^n x_i= x_1x_2 \ldots x_n$.
\begin{example}
For $n=6$, we can get $A_3$ by conjugating $I+E(3,4)$ by $(I-E(2,3))(I-E(1,2))$:
\begin{align*}
 &A_3  = \left[ \begin{matrix}
&1 & 0 & 0 & 1 & 0 & 0\\
&0 & 1 & 0 & 1 & 0 & 0\\
&0 & 0 & 1 & 1 & 0 & 0\\
&0 & 0 & 0 & 1 & 0 & 0\\
&0 & 0 & 0 &  0 & 1 & 0 \\
& 0 &0 & 0 & 0 &  0 & 1 
\end{matrix}  \right]=
 \left[ \begin{matrix}
&1 & 1 & 0 & 0 & 0 & 0\\
&0 & 1  &0& 0 & 0 & 0\\
&0 & 0 & 1 & 0 & 0 & 0\\
&0 & 0 & 0 & 1 & 0 & 0\\
&0 & 0 & 0 &  0 & 1 & 0 \\
& 0 &0 & 0 & 0 &  0 & 1 
\end{matrix}  \right] \left[ \begin{matrix}
&1 & 0 & 0 & 0 & 0 & 0\\
&0 & 1  &1& 0 & 0 & 0\\
&0 & 0 & 1 & 0 & 0 & 0\\
&0 & 0 & 0 & 1 & 0 & 0\\
&0 & 0 & 0 &  0 & 1 & 0 \\
& 0 &0 & 0 & 0 &  0 & 1 
\end{matrix}  \right] \cdot\\
& \left[ \begin{matrix}
&1 & 0 & 0 & 0 & 0 & 0\\
&0 & 1 & 0 & 0 & 0 & 0\\
&0 & 0 & 1 & 1 & 0 & 0\\
&0 & 0 & 0 & 1 & 0 & 0\\
&0 & 0 & 0 &  0 & 1 & 0 \\
& 0 &0 & 0 & 0 &  0 & 1 
\end{matrix}  \right]
\left[ \begin{matrix}
&1 & 0 & 0 & 0 & 0 & 0\\
&0 & 1  &-1& 0 & 0 & 0\\
&0 & 0 & 1 & 0 & 0 & 0\\
&0 & 0 & 0 & 1 & 0 & 0\\
&0 & 0 & 0 &  0 & 1 & 0 \\
& 0 &0 & 0 & 0 &  0 & 1 
\end{matrix}  \right] 
\left[ \begin{matrix}
&1 & -1& 0 & 0 & 0 & 0\\
&0 & 1  &0& 0 & 0 & 0\\
&0 & 0 & 1 & 0 & 0 & 0\\
&0 & 0 & 0 & 1 & 0 & 0\\
&0 & 0 & 0 &  0 & 1 & 0 \\
& 0 &0 & 0 & 0 &  0 & 1 
\end{matrix}  \right]
\end{align*}
\end{example}

We now explain why the martix
$$ \left( \prod_{j=1}^{i-2 } \left(I-(a_j-1)E(j, j+2)\right) \right)^{-1}A_i\prod_{j=1}^{i-2 }\left(I-(a_j-1)E(j, j+2)\right),$$
has the same $i+1$ column as $B$.
Conjugating $A_i$ by $I-(a_1-1)E(1, 3)$ turns the $1$ in position $(1,i+1)$ into $a_1$ in $O(\sqrt{p})$ steps, as explained in Lemma \ref{root} and by the following computation:
\begin{align*}
& \left[ \begin{matrix}
&1 & 0 & 0 & a_1 & 0 & 0\\
&0 & 1 & 0 & 1 & 0 & 0\\
&0 & 0 & 1 & 1 & 0 & 0\\
&0 & 0 & 0 & 1 & 0 & 0\\
&0 & 0 & 0 &  0 & 1 & 0 \\
& 0 &0 & 0 & 0 &  0 & 1 
\end{matrix}  \right]= \\
&\left[ \begin{matrix}
&1 & 0 & a_1-1 & 0 & 0 & 0\\
&0 & 1 & 0 & 0 & 0 & 0\\
&0 & 0 & 1 & 0 & 0 & 0\\
&0 & 0 & 0 & 1 & 0 & 0\\
&0 & 0 & 0 &  0 & 1 & 0 \\
& 0 &0 & 0 & 0 &  0 & 1 
\end{matrix}  \right]\left[ \begin{matrix}
&1 & 0 & 0 & 1 & 0 & 0\\
&0 & 1 & 0 & 1 & 0 & 0\\
&0 & 0 & 1 & 1 & 0 & 0\\
&0 & 0 & 0 & 1 & 0 & 0\\
&0 & 0 & 0 &  0 & 1 & 0 \\
& 0 &0 & 0 & 0 &  0 & 1 
\end{matrix}  \right]
\left[ \begin{matrix}
&1 & 0 & -a_1+1 & 0 & 0 & 0\\
&0 & 1 & 0 & 0 & 0 & 0\\
&0 & 0 & 1 & 0 & 0 & 0\\
&0 & 0 & 0 & 1 & 0 & 0\\
&0 & 0 & 0 &  0 & 1 & 0 \\
& 0 &0 & 0 & 0 &  0 & 1 
\end{matrix}  \right]
\end{align*}
 Conjugating by elements of the form  $I-(a_j-1)E(j, j+2), j= 2,3\ldots i-2$ and at the end multiplying from the left by $I-(a_{i-1}-1)E(i-1, i+1)$ builds the first column of the box in $O(n \sqrt{p})$ steps. Conjugating doesn't change the parity of the length. Lemma \ref{root} says that multiplying by elements of the form 
$I-(a_{i-1}-1)E(i-1, i+1)$ doesn't affect the parity of the word either. That is so far we have a word of odd length.
\begin{example}\label{fcol}
The following calculation illustrates how we conjugate $A_3$ by $\prod_{j=1}^{i-2 }I-(a_j-1)E(j, j+2),$ to create the the column elements that we desire.
\begin{align*}
 \left[ \begin{matrix}
&1 & 0 & 0 & 2 & 0 & 0\\
&0 & 1 & 0 & 3 & 0 & 0\\
&0 & 0 & 1 & 1 & 0 & 0\\
&0 & 0 & 0 & 1 & 0 & 0\\
&0 & 0 & 0 &  0 & 1 & 0 \\
& 0 &0 & 0 & 0 &  0 & 1 
\end{matrix}  \right]= 
& \left[ \begin{matrix}
&1 & 0 & 0 & 0 & 0 & 0\\
&0 & 1 & 0 & 2 & 0 & 0\\
&0 & 0 & 1 & 0 & 0 & 0\\
&0 & 0 & 0 & 1 & 0 & 0\\
&0 & 0 & 0 &  0 & 1 & 0 \\
& 0 &0 & 0 & 0 &  0 & 1 
\end{matrix}  \right]
\left[ \begin{matrix}
&1 & 0 & 1 & 0 & 0 & 0\\
&0 & 1 & 0 & 0 & 0 & 0\\
&0 & 0 & 1 & 0 & 0 & 0\\
&0 & 0 & 0 & 1 & 0 & 0\\
&0 & 0 & 0 &  0 & 1 & 0 \\
& 0 &0 & 0 & 0 &  0 & 1 
\end{matrix}  \right]\\
&\left[ \begin{matrix}
&1 & 0 & 0 & 1 & 0 & 0\\
&0 & 1 & 0 & 1 & 0 & 0\\
&0 & 0 & 1 & 1 & 0 & 0\\
&0 & 0 & 0 & 1 & 0 & 0\\
&0 & 0 & 0 &  0 & 1 & 0 \\
& 0 &0 & 0 & 0 &  0 & 1 
\end{matrix}  \right]
\left[ \begin{matrix}
&1 & 0 & -1 & 0 & 0 & 0\\
&0 & 1 & 0 & 0 & 0 & 0\\
&0 & 0 & 1 & 0 & 0 & 0\\
&0 & 0 & 0 & 1 & 0 & 0\\
&0 & 0 & 0 &  0 & 1 & 0 \\
& 0 &0 & 0 & 0 &  0 & 1 
\end{matrix}  \right]\left[ \begin{matrix}
&1 & 0 & 0 & 0 & 0 & 0\\
&0 & 1 & 0 & -2 & 0 & 0\\
&0 & 0 & 1 & 0 & 0 & 0\\
&0 & 0 & 0 & 1 & 0 & 0\\
&0 & 0 & 0 &  0 & 1 & 0 \\
& 0 &0 & 0 & 0 &  0 & 1 
\end{matrix}  \right]
\end{align*}
\end{example}

Now conjugating by $I+b_3E(i+1,i+3)$ will force the third  column of the box to be exactly what we want. Continue conjugating by elements of the form $I+b_{j+2}E(j,j+2), j= i+1, i+3, \ldots$ to create the odd columns of the box in $O(n \sqrt{p})$ steps. 
\begin{example}
This example shows how to create the third column of the box in in $O( \sqrt{p})$ steps.
\begin{align*}
&\left[ \begin{matrix}
&1 & 0 & 0 & 2 & 0 & 10\\
&0 & 1 & 0 & 3 & 0 & 15\\
&0 & 0 & 1 & 1 & 0 & 5\\
&0 & 0 & 0 & 1 & 0 & 0\\
&0 & 0 & 0 &  0 & 1 & 0 \\
& 0 &0 & 0 & 0 &  0 & 1 
\end{matrix}  \right]= \\
& \left[ \begin{matrix}
&1 & 0 & 0 & 0 & 0 & 0\\
&0 & 1 & 0 & 0 & 0 & 0\\
&0 & 0 & 1 & 0 & 0 & 0\\
&0 & 0 & 0 & 1 & 0 & -5\\
&0 & 0 & 0 &  0 & 1 & 0 \\
& 0 &0 & 0 & 0 &  0 & 1 
\end{matrix}  \right]
\left[ \begin{matrix}
&1 & 0 & 0 & 2 & 0 & 0\\
&0 & 1 & 0 & 3 & 0 & 0\\
&0 & 0 & 1 & 1 & 0 & 0\\
&0 & 0 & 0 & 1 & 0 & 0\\
&0 & 0 & 0 &  0 & 1 & 0 \\
& 0 &0 & 0 & 0 &  0 & 1 
\end{matrix}  \right]
\left[ \begin{matrix}
&1 & 0 & 0 & 0 & 0 & 0\\
&0 & 1 & 0 & 0 & 0 & 0\\
&0 & 0 & 1 & 0 & 0 & 0\\
&0 & 0 & 0 & 1 & 0 & 5\\
&0 & 0 & 0 &  0 & 1 & 0 \\
& 0 &0 & 0 & 0 &  0 & 1 
\end{matrix}  \right]
\end{align*}
\end{example}

\paragraph*{Building up the even columns.}
The next step creates the even columns in a separate, new matrix $B_2$. We begin with building up $A_{i+1}$, (see Definition\ref{A}), following the construction of \eqref{conA}. Imitate the construction of the first column to create the second column of the box as wished.   
Then multiply with $I-E(i+1,i+2)$ to get rid of the $1$ in position $(i+1,i+2)$. And then conjugate by $I+b_jE(j+2,j+4)$ for $j\geq 1$ as many times as needed to create the even columns. This way $B_2$ is expressed as a word in the generators $S$ of even length of order $O(n\sqrt{p})$.
\paragraph*{Building up $B$.}
At this point we notice that $B=B_1B_2$. This gives the conjugacy class wanted in $O(n\sqrt{p})$ steps and each generator has multiplicity at most $O(\sqrt{p})$.

The following example illustrates how an element of the conjugacy class of $S$ can be obtained if the even columns and the odd columns are constructed in two separate matrices.
\begin{example} In this example, we illustrate why the product of $B_1$ and $B_2$ is $B$.
\begin{align*}
& \left[ \begin{matrix}
&1 & 0 & 0 & 2 & -4 & 10\\
&0 & 1 & 0 & 3 & -6 & 15\\
&0 & 0 & 1 & 1 & -2 & 5\\
&0 & 0 & 0 & 1 & 0 & 0\\
&0 & 0 & 0 &  0 & 1 & 0 \\
& 0 &0 & 0 & 0 &  0 & 1 
\end{matrix}  \right]= 
 \left[ \begin{matrix}
&1 & 0 & 0 & 2 & 0 & 10\\
&0 & 1 & 0 & 3 & 0 & 15\\
&0 & 0 & 1 & 1 & 0 & 5\\
&0 & 0 & 0 & 1 & 0 & 0\\
&0 & 0 & 0 &  0 & 1 & 0 \\
& 0 &0 & 0 & 0 &  0 & 1 
\end{matrix}  \right]
& \left[ \begin{matrix}
&1 & 0 & 0 & 0 & -4 & 0\\
&0 & 1 & 0 & 0 & -6 & 0\\
&0 & 0 & 1 & 0 & -2 & 0\\
&0 & 0 & 0 & 1 & 0 & 0\\
&0 & 0 & 0 &  0 & 1 & 0 \\
& 0 &0 & 0 & 0 &  0 & 1 
\end{matrix}  \right]
\end{align*}
\end{example}
\begin{remark}
The construction presented in this section does not work for the conjugacy class of $C_1(\pm 1)$. But in this case, bulding up the first column of the box is easy, because we can begin with $A_1=I\pm E(1,2)$, which has the same first column as $B$. Then we continue building $B_1 $ as explained directly above Example \ref{fcol}. $B_2$ is built on the same way as before.
\end{remark}

\subsection{Building up the conjugacy class of $C_i( \pm a)$}
In this section, we explain how to express an element of $C_i(\pm a)$ as a word in the generators $S$.
\begin{lemma}\label{C_i(a)}
Let $B \in C_i(\pm a)$. Then $B$ can be expressed as a word in the generators $S$, whose length is odd and is at most $O(n\sqrt{p})$. Also, each generator $I \pm E(j,j+1)$ appears at most $O(\sqrt{p})$ times in such a word.
\end{lemma}

\begin{proof}
To construct the conjugacy class of $I+ a E(i,i+1)$, things are similar to Section \ref{first paths}. Let $B \in C_i( \pm a)$. We start by constructing $A_i$, as indicated by \eqref{conA} in at most $2n+1$ steps. Then we create the entries in positions $(j,i+1), j \in \{1,\ldots, i-1\}$, just as described directly above Example \ref{fcol}. Then, we multiply by $I+ (a-1) E(i,i+1)$ from the left to set the entry on position $(i,i+1)$ equal to $a$. This is illustrated as
\begin{align*}
& \left[ \begin{matrix}
&1 & 0 & 0 & a_1 & 0 & 0\\
&0 & 1 & 0 & a_2 & 0 & 0\\
&0 & 0 & 1 & a & 0 & 0\\
&0 & 0 & 0 & 1 & 0 & 0\\
&0 & 0 & 0 &  0 & 1 & 0 \\
& 0 &0 & 0 & 0 &  0 & 1 
\end{matrix}  \right]= (I+ (a-1) E(i,i+1) )  \left[ \begin{matrix}
&1 & 0 & 0 & a_1 & 0 & 0\\
&0 & 1 & 0 & a_2 & 0 & 0\\
&0 & 0 & 1 & 1 & 0 & 0\\
&0 & 0 & 0 & 1 & 0 & 0\\
&0 & 0 & 0 &  0 & 1 & 0 \\
& 0 &0 & 0 & 0 &  0 & 1 
\end{matrix}  \right].
\end{align*}
Now conjugate by $(I+ E(i,i+2))^{\frac{b_2}{a}}= I+ \frac{b_2}{a} E(i,i+2)$ to get exactly the third column. Notice that the third column of the box occurs by multiplying the first column of the box by $\frac{b_2}{a}$.
\begin{align*}
& \left[ \begin{matrix}
&1 & 0 & 0 & a_1 & 0 & a_1b_2a^{-1} \\
&0 & 1 & 0 & a_2 & 0 & a_2b_2a^{-1} \\
&0 & 0 & 1 & a & 0 &b_2 \\
&0 & 0 & 0 & 1 & 0 & 0\\
&0 & 0 & 0 &  0 & 1 & 0 \\
& 0 &0 & 0 & 0 &  0 & 1 
\end{matrix}  \right]= (I+\frac{b_2}{a}  E(i,i+2) )   \left[ \begin{matrix}
&1 & 0 & 0 & a_1 & 0 & 0\\
&0 & 1 & 0 & a_2 & 0 & 0\\
&0 & 0 & 1 & a & 0 & 0\\
&0 & 0 & 0 & 1 & 0 & 0\\
&0 & 0 & 0 &  0 & 1 & 0 \\
& 0 &0 & 0 & 0 &  0 & 1 
\end{matrix}  \right] (I-\frac{b_2}{a} E(i,i+2) ).
\end{align*}
Conjugating by $(I-\frac{b_4}{b_2} E(i+2,i+4) ),$ we get the the fifth column. Continuing this way, we construct $B_1$, the matrix of the odd columns. $B_2$ is constructed just like in Section \ref{first paths}. Therefore, we have that $B=B_1B_2$. 
Notice that again, the length of the word is at most $O(n\sqrt{p})$ and the multiplicity of each generator is at most $O(\sqrt{p})$.

\end{proof}
\begin{remark}
The construction presented in this section does not work for the conjugacy class of $C_1(\pm a)$. But in this case, bulding up the first column of the box is easy, because we can begin with $A_1=I\pm a E(1,2)$, which has the same first column as $B$. Then we continue building $B_1 $ as explained directly above Example \ref{fcol}. $B_2$ is built on the same way as before.
\end{remark}

\section{Proof of Theorem \ref{main}.}\label{proofmain}
\begin{proof}
Lemmas \ref{comp}, \ref{C_i(1)} and \ref{C_i(a)} prove that $A= O( p n^2) $. We consider $A=kpn^2$, where $k$ is a universal constant. Therefore, Lemma \ref{comp} gives that there exist universal constants $0<b,d< \infty$ such that  for $c>2$ and $t\geq c b p^2n^4$, we have that
$$|G|||P^{*t}- \mu||^2_2 \leq |G| \left(  e^{-t/A}+ ||Q^{*t/2A}- \mu||^2_2 \right) \leq $$
$$ p^{n^2}e^{-c k n^2p}+\alpha e^{-c}\leq d e^{-c},$$
since $p^{n^2}e^{-c k n^2p} \leq e^{-c}$.
\end{proof}

\section{Acknowledgments}
I would like to thank Persi Diaconis for all the useful comments and guidance he provided. 

This paper is adapted from my PhD thesis, for the completion of which I was supported by Stanford University.
\bibliographystyle{plain}
\bibliography{second_version}
\end{document}